\documentclass[11pt]{article}

\usepackage[utf8]{inputenc}
\usepackage[T1]{fontenc}
\usepackage{tgpagella}
\usepackage{amssymb}
\usepackage{amsfonts}
\usepackage{amsmath}
\usepackage{amsthm}
\usepackage{amscd}
\usepackage{bussproofs}
\usepackage[all]{xy}
\usepackage{comment}

\newtheorem{tw}{Theorem} 
\newtheorem{wniosek}[tw]{Corollary} 
\newtheorem{lemat}[tw]{Lemma} 
\newtheorem{fakt}[tw]{Fact} 

\theoremstyle{definition}
\newtheorem{definicja}[tw]{Definition} 
\newtheorem{konwencja}[tw]{Convention} 

\newtheorem{przyklad}[tw]{Example} 
\newtheorem{uwaga}[tw]{Remark} 

\newcommand{\qcr}[1]{\ulcorner #1 \urcorner}
\newcommand{\df}[1]{\textbf{#1}}
\newcommand{\Prov}{\textnormal{Pr}}
\newcommand{\CT}{\textnormal{CT}}
\newcommand{\PA}{\textnormal{PA}}
\newcommand{\Sent}{\textnormal{Sent}}
\newcommand{\Subst}{\textnormal{Subst}}
\newcommand{\val}[1]{({#1})^{\circ}}
\newcommand{\len}{\textnormal{len}}

\newcommand{\TermSeq}{\textnormal{TermSeq}}

\newcommand{\conc}{^\frown}
\newcommand{\set}[2]{\{#1 \ \ | \ \ #2 \}}

\newcommand{\term}{\textnormal{Term}}
\newcommand{\form}{\textnormal{Form}}

\newcommand{\Th}{\textnormal{Th}}

\newcommand{\was}[1]{\{#1\}}
\newcommand{\Seq}[1]{\langle #1 \rangle}

\newcommand{\PRA}{\textnormal{PRA}}
\newcommand{\Ax}{\textnormal{Ax}}

\newcommand{\num}[1]{\underline{#1}}

\title{Notes on bounded induction for the compositional truth predicate}
\author{Bartosz Wcisło, Mateusz Łełyk}

\begin{document}

\maketitle

	\begin{abstract}
	We prove that the theory of the extensional compositional truth predicate for the language of arithmetic with $\Delta_0$-induction scheme for the truth predicate and the full arithmetical induction scheme is not conservative over Peano Arithmetic. In addition, we show that a slightly modified theory of truth actually proves the global reflection principle over the base theory.
	\end{abstract}
	
\section{Introduction}

This paper concerns conservativeness of the compositional truth predicate with bounded induction over Peano Arithmetic. We say that a theory $Th_1$ is \df{conservative} over a theory $Th_2$ with respect to a class of formulae $\Gamma\subseteq \Sent_{\mathcal{L}_{Th_2}}$ (or simply $\Gamma$-conservative) iff for every sentence $\phi \in \Gamma$ if $Th_1 \vdash \phi,$ then $Th_2 \vdash \phi.$ If $\Gamma$ is the whole class of first-order formulae over the language of the theory $Th_2,$ then we simply say that $Th_1$ is \df{conservative} over $Th_2.$ If $\Gamma$ happens to be the class of all formulae in the language of Peano Arithmetic, then we say that $Th_1$ is \df{arithmetically conservative} over $Th_2.$ Verifying various conservativeness results for theories of interest forms an established line of research. It is important from both philosophical and purely logical point of view. From a philosophical point of view, it might be argued that conservativeness of $Th_1$ over $Th_2$ assures that accepting the axioms of the former theory 
does not force us to make any new commitments as to what is actually the case than accepting the latter. This motivation is particularly important in case of the truth predicate, whose triviality is extensively discussed in contemporary philosophy. Namely, the adherents of the  \df{deflationary theory of truth} claim that the truth predicate does not have any actual content but is rather a purely logical device. This vague and imprecise claim has been explicated by some critics (most notably Horsten in \cite{horsten}, Shapiro in \cite{shapiro} and Ketland in \cite{ketland}) in terms of conservativeness. Were the truth predicate void of substantial content, whatever that would exactly mean, the axioms governing it should not allow us to derive more arithmetical theorems than our theory of arithmetic alone.  Thus conservativeness results help to clarify the picture and to distinguish the principles of truth that account for it having a substantial content. From the point of view of pure logic, conservativeness is one of the very basic relations between the theories of interest. It shows which principles may be added to a given theory without the risk of inconsistency or inadequacy.  In this context it is important that in practice most conservativeness results are possible to obtain in weak theories like $\PRA$  or $I\Sigma_1$.\footnote{Theorem 5.2 in \cite{enayatvisser} states that the conservativity proof for $\CT^-$ may be carried out in $\PRA.$ The same argument is valid in the case of $\CT^- $ with the internal induction for the arithmetical formulae. In \cite{leigh} (Theorem 2) it is shown that $\CT^-$ and $\CT^- $ with the internal induction for the arithmetical formulae are conservative over $\PA$ provably in $I \Delta_0 + \exp_1$, where $\exp_1$ denotes the hyper-exponential function.}  Moreover verifying that one theory is conservative over another might deliver new information about the strength of the latter, especially when the first one allows us to employ notions and principles which are \textit{prima facie} not available in the second one.


The situation was particularly interesting in the case of the compositional truth predicate over $\PA.$ Let us first define what we precisely mean by this notion. Before that, let us introduce a few notational conventions. We assume that the reader is familiar with the arithmetization of syntax as explained, e.g. in \cite{kaye}. 

\begin{konwencja}\label{konw: glupia konwencja}
	
	\begin{itemize}
		
		\item We assume that we have some fixed G\"odel coding. By $\qcr{x}$ we denote the G\"odel code of $x$. We will also use $\qcr{x}$ to represent the numeral for the G\"odel code of $x$.
		\item ''$y=\num{x}$'' is an arithmetical formula representing a natural primitive recursive function assigning to an element $x$ the code of its numeral, i.e. $\qcr{SS\ldots S(0)},$ where the successor function symbol $S$ occurs $x$ times;
		\item $\term (x)$ is some fixed arithmetical formula representing the set of (G\"odel codes of) arithmetical terms;
		\item the arithmetical formula $\form(x)$ represents the set of (G\"odel codes of) arithmetical formulae, the formula $\Sent(x)$ represents the set of (G\"odel codes of) arithmetical sentences, i.e. formulae with no free variables;
		\item the arithmetical formula $\Ax_{\PA}(x)$ represents the set of axioms of Peano Arithmetic;
		\item by a proof in the sequent calculus from some set of axioms $\Gamma$ we mean a proof in the sequent calculus, where the additional initial sequents ''$\longrightarrow \phi$'' are allowed, where $\phi \in \Gamma$. In particular, by a proof of $\phi$ from $\PA$ in the sequent calculus we mean an ordinary proof of $\phi$ from axioms of $\PA$ in sequent calculus for first-order logic;\footnote{Rather than a proof of $\phi$ from the empty set of premises in sequent calculus with the additional induction rule or with the $\omega$-rule.}
		\item  by $\Prov_{\Th}(x)$ we mean an arithmetical formula formalising the unary relation: ''there exists (a code of) a proof $d$ of the sentence $x$ in the sequent calculus from the axioms of the theory $\Th$'';		 
		\item if $\tau(y)$ is some formula, then by $\Prov_{\tau}(x)$ we mean an arithmetical formula formalising the relation: ''there exists (a code of) a proof $d$ of the sentence $x$ in the sequent calculus, where the initial sequents of the form '$\longrightarrow \phi$' are allowed for $\phi$ such that $\tau(\phi)$''. Since in our applications $\tau(y)$ will be thought of as some form of a truth predicate, $\Prov_{\tau}(x)$ reads as: ''there exists a proof of $x$ from true premises'';  
		\item  $\Subst(\phi(v), t)$ is an arithmetical formula representing the primitive recursive substitution function, which assigns to a (code of a) formula $\phi(v)$ with at most one free variable and (a code of a) term $t$ the unique (code of the) sentence resulting from substituting $t$ for every free occurrence of $v$ in $\phi$. Additionally we assume that $\Subst$ is the identity function whenever applied to sentences (i.e., whenever there is no free variable in $\phi$);
		\item the arithmetical formula  ''$y=t^{\circ}$'' represents a natural primitive recursive function assigning to each (code of a) term $t$ its value (and undefined if $t$ is not a code of a term); 
		\item  ''$x \in y$'' is an arithmetical formula expressing that $x$-th bit of the binary expansion of $y$ is $1$ and we write that $ x \notin y$ iff either it is $0$ or $y < 2^x;$ 
		\item let $M$ be an arbitrary model of a signature expanding the language of $\PA.$ Let $I \subset M$ be an arbitrary set. We say that $I$ is \df{coded} in $M$ iff there exists an element $c$ such that $x \in I$ iff $M \models x \in c.$ We call that $c$ the \df{code of the set} $I$. We will sometimes identify the subset $I$ with its code $c$.\footnote{There is a difference between the above definition of a coded set, and the notion of a \emph{coded subset of $\omega$}. Usually we say that a subset $A \subseteq \omega$ is coded in a model $M \supset \omega$ iff there exists $c$ such that $\set{x \in M}{M \models x \in c} \cap \omega = A.$ The set $A$ will not, in general, be coded in any model of $\PA$ in our sense of coding.} 
		
	\end{itemize}
\end{konwencja}

We will usually suppress the distinction between syntactical objects and their codes, e.g. we will sometimes write, e.g. $\Prov_{\Th}(0=0)$ instead of $\Prov_{\Th}(\qcr{0=0}).$ In order to avoid any confusion let us introduce a few more notational conventions.

\begin{konwencja} \
	
	\begin{itemize} 
		\item By using variables $\phi, \psi$ we implicitly restrict quantification to (G\"odel codes of) arithmetical sentences. I.e. by $\forall \phi \ \ \Psi( \phi)$ we mean $\forall x \ \ \Sent(x) \rightarrow \Psi(x)$ and by $\exists \phi \ \ \Psi(\phi)$ we mean $\exists x \ \ \Sent(x) \wedge \Psi(x).$ For brevity we will sometimes also use variables $\phi, \psi$ to run over arithmetical formulae, whenever it is clear from the context which one we mean; 
		\item similarly, $\phi(v), \psi(v)$ run over arithmetical formulae with at most one indicated free variable (i.e. $\phi(v)$ is either a formula with exactly one free variable or a sentence); 
		\item  $s,t$ run over codes of closed arithmetical terms;
		\item $v,v_1,v_2, \ldots, w, w_1, w_2, \ldots$ run over codes of variables;
		\item  to enhance readability we suppress the formulae representing the syntactic operations. We will usually be writing the results of these operations instead, e.g.  $\Phi(\psi \wedge \eta)$ instead of $\Phi(x) \wedge$ ''$x$ is the conjunction of $\psi$ and $\eta$'' similarly, we write $\Phi(\psi(t))$ instead of $\Phi(x) \wedge \textnormal{Subst}(\psi,t)$;	
		\item Let $M$ be an arbitrary model of $\PA$.  If $x \in M$ is greater than $k$ for any number $k \in \omega,$ then we call it \df{nonstandard}. We call it \df{standard} otherwise. Since we often ignore the difference between syntactical objects and their G\"odel codes, we will often say that a sentence, formula or term is nonstandard meaning that its code $x$ satisfies in $M$ all the sentences $x>k$ for natural numbers $k$. 
	\end{itemize}			
\end{konwencja}

\begin{definicja}\label{CT}
	By $\CT^-$ we mean the theory obtained by extending the arithmetical signature with an additional predicate $T(x)$ (with the intended reading ''$x$ is a G\"odel code of a true sentence'') and extending axioms of $\PA$ with the following ones:
	\begin{enumerate}
		\item $\forall s,t \ \Bigl( T(s=t) \equiv s^{\circ}=t^{\circ} \Bigr).$
		\item $\forall \phi, \psi \ \Bigl( T(\phi \otimes \psi) \equiv T \phi \otimes T \psi \Bigr).$
		\item $\forall \phi \ \Bigl( T (\neg \phi) \equiv \neg T \phi \Bigr).$
		\item $\forall v, \phi(v) \ \Bigl( T (Q v \ \phi(v)) \equiv Q x \ T(\phi(\num{x})) \Bigr).$ 	
	\end{enumerate}				
	
	Here $\otimes \in \{ \wedge, \vee \}$ and $Q \in \{\forall, \exists \}.$
\end{definicja}

Note that we in $\CT^-$ \emph{do not assume} that any formulae with the truth predicate  satisfy the induction scheme. 

\begin{uwaga}
	Note that the compositionality axioms for quantifiers which we have listed above are \emph{not} exactly the ones typically assumed in the definition of compositional theories of truth with no induction. For simplicity, let us discuss this  difference in a specific case of the universal quantifier, although analogous remarks apply to the existential quantifier as well. A standard axiom as given e.g. in \cite{halbach} has the following form:
	
	\begin{displaymath}
	\forall v, \phi(v) \ \ T(\forall v \ \phi(v)) \equiv \forall t \ T(\phi(t)). \tag*{$(*)$}
	\end{displaymath}
	
	According to the above axiom a universal formula $\forall v \ \phi(v)$ is true only if for arbitrary term $t$ its substitutional instance $\phi(t)$ is true, which is not the same as to say that for arbitrary \emph{numeral} $\num{x}$ the formula $\phi(\num{x})$ is true. In the presence of $\Sigma_1$-induction for the formulae containing truth predicate both versions of the quantifier axioms are equivalent, since in such a case we may prove the principle of extensionality (or, as we will also sometimes call it, \df{regularity}), i.e. the following sentence:
	
	\begin{equation*} \label{regularity}
	\forall \phi \forall t,s\ \ \bigl(  s^{\circ} = t^{\circ} \rightarrow  T\phi(t)\equiv T\phi(s)\bigr). \tag{\textnormal{REG}}
	\end{equation*}
	
	The principle states that the truth-value of a formula  does not depend on specific terms which occur in the formula, but rather on their values (and $\PA$ proves that for every term there exists a numeral with the same value). It is arguable however, whether the axiom $(*)$ as stated above is really intuitive under the assumption that we lack any induction for the extended language whatsoever. Presumably the initial intuition is that $\forall x \ \phi(x)$ is true iff the formula $\phi(x)$ is satisfied by all elements. If we want to work with the truth predicate rather than the satisfaction relation, we have to rethink what it means to be satisfied by all elements. One way to say it is that for \emph{an arbitrary} element $x$ the sentence $\phi(t)$ is true, where $t$ is \emph{some} term denoting $x.$ And since in systems whose intended domain are natural numbers every element $x$ is denoted by a numeral $SS \ldots S0$ (with $S$ repeated $x$ times), this intuition seems to be perfectly addressed by the axiom we opt for, i.e.:  
	
	\begin{displaymath}
	\forall v, \phi(v) \ \ T(\forall v \ \phi(v)) \equiv \forall x \ T(\phi(\num{x})). 
	\end{displaymath}
	
	On the other hand, there is the intuition that universal sentences should satisfy \textit{dictum de omni} principle, i.e. whenever $\forall v \ \phi(v)$ is true, we expect $\phi(t)$ to be true for arbitrary terms. These two intuitions for the truth of universal sentences may diverge when we lack induction for the formulae containing the truth predicate, but we see no obvious reason why substitution principle should be regarded more essential than the intuition that all natural numbers can be named by numerals.  Therefore, we do not think that the compositional axioms for quantifiers which we assume are less natural than the ones that are typically formulated. Namely, we assume that to infer that a universal sentence $\forall v \ \phi(v)$ is true, we only need to assume that $\phi(\num{x})$ holds for all numerals $\num{x}$ rather than require the stronger hypothesis to hold, namely that $\phi(t)$ holds for all \emph{terms}. Probably the most satisfactory solution, when working with such weak theories as $\CT^-$, would be to embrace both intuitions and accept asymmetric compositional axioms for the universal quantifier, i.e.:
	
	\begin{enumerate}
		\item $\forall v, \phi(v) \ \ T\bigl(\forall v \ \phi(v)\bigr) \longrightarrow \forall t \ T\bigl(\phi(t)\bigr).$
		\item $\forall v, \phi(v) \ \ \forall x \ T\bigl(\phi(\num{x})\bigr) \longrightarrow T\bigl(\forall v \ \phi(v)\bigr).$
	\end{enumerate}
	
	Of course, in this case, we should also introduce an analogous pair of axioms for the existential quantifier. We remark that, since a variant of $\CT^-$ using the above asymmetric axioms for quantifiers is stronger than the one chosen by us, our non-conservativity results apply also to the extensions of $\CT^-$ with the compositional axioms for quantifiers defined in such way.\footnote{We thank the anonymous referee for the remark that led to this discussion.}
\end{uwaga}

Let us make one more comment at the end of this section. Taking into account that the choice of the precise formulation for the quantifier axioms seems problematic, one may wonder, whether it would not be easier to prove our results for the compositional \emph{satisfaction} predicate rather than the truth predicate, since then there seems to be a canonical choice of quantifier axioms.

The basic reason why we have decided to work with the truth predicate is  rather trivial. We wanted to conform to the standard conventions in the field of the axiomatic truth theories, especially that in the context of weak theories results obtained for the satisfaction predicate do not in general carry over to the truth-theoretic framework quite automatically precisely because of the extensionality issues.

\section{Known Results on Conservativity of Extensions of $\CT^-$}

Let us list a few important theories obtained via augmenting $\CT^-$ with some induction.

\begin{definicja} \label{CT0}
	By $\CT$ we mean the theory obtained by adding to $\CT^-$ all the instances of the induction scheme (i.e. including the ones for formulae containing the truth predicate). By $\CT_1$ we mean $\CT^-$ with induction for $\Pi_1$-formulae containing the truth predicate. By $\CT_0$ we mean $\CT^-$ with induction for $\Delta_0$ formulae containing the truth predicate. 
\end{definicja}

One of the most important questions of theory of truth may be now rephrased: which reasonable extensions of $\CT^-$ are conservative over $\PA$? This question is indeed nontrivial thanks to the following theorem:\footnote {Let us explain ourselves for this complicated attribution. A related result has been obtained  by Kotlarski, Krajewski and Lachlan in \cite{kkl}, who essentially proved Theorem \ref{tw_kkl} for a variant of $\CT^-$ with \emph{satisfaction} relation in place of the truth predicate. However, it is by no means obvious to us how to modify the proof of Kotlarski, Krajewski and Lachlan so that it works for the \emph{truth} predicate. In order to do this we would apparently have to prove that the extensionality principle for the satisfaction relation is conservative over $\PA$, which does not seem trivial. Theorem \ref{tw_kkl} for the version of $\CT^-$ axiomatised in \emph{purely relational language} was proved by Enayat and Visser in \cite{enayatvisser}. The full-blown result was shown by proof-theoretic methods by Leigh in \cite{leigh}.}

\begin{tw}[Krajewski--Kotlarski--Lachlan, Enayat--Visser, Leigh] \label{tw_kkl}
	$\CT^-$ is conservative over $\PA.$
\end{tw}

This result may be still improved in a substantial way.

\begin{definicja}
	By the principle of \df{internal induction} we mean the following axiom:
	\begin{equation}\label{II}\tag{\textnormal{INT}}
	\forall \phi(v) \ \ \biggl( \forall x \ \Bigl(T (\phi(\num{x})) \rightarrow T(\phi(\num{Sx}))\Bigr) \longrightarrow \Bigl( T (\phi (\num{0})) \rightarrow \forall x  \ T (\phi(\num{x})) \Bigr) \biggr).
	\end{equation}		  
\end{definicja}

The name of this axiom is introduced in analogy to the distinction between internal and external induction axioms in subsystems of second-order arithmetic.

The proof from  \cite{enayatvisser} (as well as the one from \cite{leigh})  of conservativeness of $\CT^-$ over $\PA$ , gives as a corollary the following theorem.\footnote{Actually in the both cited papers a much more general theorem has been presented from which Theorem \ref{EV} follows as a direct corollary, see \cite{enayatvisser}, remarks in Section 6 and \cite{leigh}, Theorem 3.}

\begin{tw}[Enayat--Visser, Leigh] \label{EV}
	$\CT^- +$ \eqref{II} is conservative over $\PA.$
\end{tw}

Let us recall the regularity principle \eqref{regularity}. Instead of considering an unusual variant of $\CT^-$ we could have added the above principle to the standard list of axioms, because it yields both forms of the compositional axiom for the quantifiers equivalent. Fortunately, this would not trivialise our work, thanks to the following theorem, which may be read directly off the Enayat--Visser construction.

\begin{tw}[Enayat--Visser] \label{EVR}
	$\CT^- +$ \eqref{II} + \eqref{regularity} is conservative over $\PA.$
\end{tw}

It is easily observed that the internal induction may be proved in a system $\CT_0$ that is in $\CT^-$ with $\Delta_0$-induction for the truth predicate. Then, as we shall briefly indicate, it follows  that $\CT^-$ with $\Pi_1$ induction for the truth predicate is enough to prove the following principle: 

\begin{definicja}
	By the \df{Global Reflection Principle} we mean the following axiom:
	\begin{equation}\label{GRP}\tag{GRP}
	\forall \phi \ \Bigl( \Prov_T( \phi) \longrightarrow T (\phi) \Bigr),
	\end{equation}
where $\Prov_T(x)$ is a special case of the predicate $\Prov_{\tau}(x)$ defined in Convention \ref{konw: glupia konwencja} with $\tau(y) = T(y).$ Hence the intuitive reading of $\Prov_T(x)$ is: ''there exists a proof $d$ of the sentence $d$ in sequent calculus from the initial sequents '$\longrightarrow \phi$', where we have $T(\phi),$ i.e. a proof in sequent calculus from true premises''. 

\end{definicja}

Note that, speaking informally, \eqref{GRP} says that the set of true sentences is closed under reasonings in First-Order Logic.

\begin{definicja}
	By the \df{Axiom Soundness Property} for $\Th$ we mean the following principle:
	\begin{equation}\label{SSP} \tag{ASP}
	\forall \phi \ \Bigl( \Ax_{\Th}(\phi) \longrightarrow T (\phi) \Bigr).
	\end{equation}
\end{definicja}

Note that if a theory of truth proves both the axiom soundness of $\Th$ and satisfies the global reflection principle, then it proves the following statement:

\begin{displaymath}
\forall \phi \ \Bigl( \Prov_{\Th}(\phi) \longrightarrow T (\phi) \Bigr).
\end{displaymath}

Then a crude lower-bound for the strength of non-conservative truth-theoretic principles is given by the following theorem:

\begin{tw}
	$\CT_1$ proves the global reflection principle and the axiom soundness property for $\PA$ and thus the consistency of $\PA.$ In effect, this theory is not conservative over $\PA.$ 	
\end{tw}

The above result is obtained as follows: working in $\CT_1$, by a straightforward induction on length of proofs in the sequent calculus we show that for every substitution of closed terms for free variables in the sequent, if every formula in the antecedent is true, then some formula in the succedent of the sequent arrow is true. This is obviously a $\Pi_1$ statement, since the quantifiers ''every formula in the antecedent'', ''every formula in the succedent'' may be assumed to be bounded by the size of the proof in question (under a reasonable coding of syntax and of sets). The consistency of $\PA$ follows, since we may prove in $\CT_0$ that the parameter-free variant of induction holds and then, by $\Pi_1$ induction on the length of the block of universal quantifiers, that any universal closure of an instance of the induction scheme is true as well. 

Then a natural question arises of how to improve bounds on minimal truth-theoretic principles which, combined with $\CT^-$, are non-conservative over $\PA.$ The first natural candidate to consider is the theory $\CT_0.$ Indeed, Kotlarski in his paper \cite{kotlarski} has presented an alleged proof that the theory $\CT_0$  proves the global reflection principle.\footnote{More precisely, he considered a theory of satisfaction rather than truth, and only aimed to show that every formula derivable in $\PA$ is satisfied under all valuations.} Unfortunately, as observed independently by Albert Visser and Richard Heck, the proof contained a gap which seemed to require an essentially new approach to surmount. Actually, after the gap has been revealed, it has been completely unclear whether or not the theorem is true at all. We shall discuss the erroneous proof in the appendix. In the present paper we provide a non-conservativeness proof for our variant of $\CT_0.$ Moreover, we prove that the global reflection principle is arithmetically conservative over $\CT_0.$ In addition we show that a very natural modification of $\CT_0$ actually proves the principle.

\section{The main result}

In the paper \cite{fujimoto} Fujimoto has argued that the right way to compare the conceptual strength of theories of truth $Th_1, Th_2$ over the same base theory $B$ (in our case $B = \PA$) is to check whether $Th_1$ defines a truth predicate satisfying axioms of $Th_2.$ In such case we say that $Th_2$ is \df{$B$-relatively definable}\footnote{Fujimoto calls this relation relative \textit{truth} interpretability (and keeps the parameter $B$ implicit).}
in $Th_1.$ The relative definability is a very natural and strong relation between two theories. In particular, if $\Th_2$ is relatively definable in $\Th_1$, it implies the following other relations\footnote{All the  three points have been observed in the original paper \cite{fujimoto}, see p. 324 for the first two of them and  Proposition 28 (1) for the third one.} (on the other hand none of the these relations implies that $\Th_2$ is relatively interpretable in $\Th_1$):

\begin{enumerate}
	\item $\Th_2$ is arithmetically conservative over $\Th_1$.
	\item $\Th_2$ is interpretable is $\Th_1.$ 
	\item Every model of $\Th_1$ expands to a model of $\Th_2$.

\end{enumerate}

Let us briefly comment upon the first of the three points. Suppose that $\Th_1$ relatively interprets $\Th_2$, i.e. there is a formula $\tau(x)$ which provably  satisfies the truth axioms of $\Th_2$. Consider any proof in $\Th_2$ with an arithmetical consequence $\phi.$ Then replace all occurrences of the truth predicate in that proof with the formula $\tau(x)$ and precede all the uses of $\Th_2$'s truth-theoretic axioms with a proof in $\Th_1$ that $\tau(x)$ indeed satisfies the axioms that we used. In such a way we can obtain a proof of $\phi$ in $\Th_1$.

 We show that in this sense $\CT_0$ with the global reflection principle and the axiom soundness property is as strong as $\CT_0$ alone.

\begin{tw} \label{tw_niekonserwatywnosc_ct0}
	$\CT_0$ with the global reflection axiom and the axiom soundness property for $\PA$  is $\PA$-relatively definable in $\CT_0.$ In particular $\CT_0$ is not conservative over $\PA.$
\end{tw}

In other words, there is a formula $T'(x)$ such that in $\CT_0$ one can prove that $T'(x)$ satisfies compositional conditions for arithmetical sentences, \linebreak global reflection principle and axiom soundness property for $\PA$. Moreover, we can prove in $\CT_0$ every instance of $\Delta_0$-induction scheme for the predicate $T'(x)$.

Before we present the proof of the above theorem in full detail, let us give a sketch of our argument. This outline will be imprecise and not fully correct, but the rest of the section will be generally devoted to spelling out all the details in a proper manner.

 One of the main tools in metamathematics are various forms of partial truth predicates. Probably the best known of them are arithmetical truth predicates for the classes $\Sigma_n$.\footnote{For a definition see e.g. \cite{hajekpudlak} Definitions 1.71 and 1.74. See \cite{hajekpudlak} pp. 50--61 for a general discussion.}  Consider a partial arithmetical truth predicate $\tau(x)$. Its main feature is that for some formulae we have:

\begin{displaymath}
\tau(\qcr{\phi}) \equiv \phi.
\end{displaymath}

If $\tau(x)$ is an arithmetical formula and the truth predicate $T(x)$ satisfies $\CT^-$, then this entails that for standard sentences $\phi$ we have:

\begin{displaymath}
T\tau(\num{\phi}) \equiv T\phi.
\end{displaymath} 

Now, it may be hoped that we can define some families of formulae formulae $T_c(v)$ (or, more precisely, their codes, so that the parameter $c$ may be nonstandard), which behave like truth predicates in the sense that $\CT_0$ may prove:

\begin{displaymath}
T  T_c(\num{\phi})  \equiv T\phi
\end{displaymath}

for all arithmetical formulae $\phi$ with codes smaller than $c$. We could wish to do still better and to require that the newly defined truth predicates are compositional for small enough formulae in the sense that we have e.g.

\begin{displaymath}
T T_c( \num{\phi \wedge \psi}) \equiv T T_c(\num{\phi}) \wedge T T_c({\num{\psi}}).
\end{displaymath} 

Now, the key idea is that in the presence of $\Delta_0$ induction for our original truth predicate $T(x)$ we may hope for a full induction for the ''truth predicates'' $TT_c(\num{x}).$ Before we explain why this should be case, let us introduce some notation, which will also be used in the proof proper. 

First of all, we define a formula $T'_c(x)$ as $TT_c(\num{x})$. The next notational convention deserves a separate definition.

\begin{definicja}
	Let $P(u)$ be a fresh\footnote{That is, neither from the arithmetical signature, nor $T(u)$.} unary predicate with the only variable $u$, $\delta(x)$ an arbitrary formula and $\psi$ another formula with exactly one free variable, possibly with parameters. Then by
	\begin{displaymath}
	\delta[\psi/P](x)
	\end{displaymath}
	we mean the effect of formally substituting the formula $\psi(t)$ for all occurences of $P(t)$, where $t$ is an arbitrary term, possibly changing the names of the bounded variables so as to avoid clashes. 
	
\end{definicja}

Probably, the above definition is best understood via an example. 

\begin{przyklad}
	Let $\delta(x) = \exists z, w \ \ (P(z+w) \wedge \forall z<w \neg P(z)) \wedge (x=x) \wedge P(x).$ Then 
	\begin{displaymath}
	\delta[(u>u)/P](x) = \exists z, w \ \ (z+w>z+w \wedge \forall z<w \ (\neg z>z)) \wedge (x=x) \wedge x>x.
	\end{displaymath}
	
\end{przyklad}

It is clear that the formula substitution operation may be formalized in $\PA$. 

Let us sketch, why the predicates $T'_c(x) := TT_c(\num{x})$ satisfy full induction scheme. First, one can check that the principle of internal induction \eqref{II} may be proved by a straightforward application of $\Delta_0$-induction. Then we would like to show that for arbitrary arithmetical $\phi$ we have: 

\begin{displaymath}
\Bigl(\forall x \  \bigl(\phi[T'_c/P](x)  \rightarrow \phi[T'_c/P](Sx) \bigr)\Bigr) \longrightarrow \Bigl(  \phi[T'_c/P] (0) \rightarrow \forall x  \  \phi[T'_c/P](x) \Bigr).
\end{displaymath}

 So let us fix any arithmetical formula $\phi$ and consider the formula \linebreak $T(\phi [T_c/P])(x).$ Now, since $T$ is compositional we can push it down the syntactic tree of the (nonstandard, but arithmetical) formula $\phi[T_c/P]$ finitely many levels, until it meets a partial truth predicate $T_c,$ but not any further. Since we assumed that $\phi$ is of standard syntactic shape we do not need any induction to do that. We obtain the following equivalence, which we call \df{the generalised commutativity} principle (strictly speaking, we only get something resembling it --- see the further comments):

\begin{equation*} \label{GC} \tag{GC}
T(\phi[T_c/P](x)) \equiv \phi[TT_c/P](x).
\end{equation*}

Now, by the internal induction principle we have:

\begin{multline*}
\forall x \  \biggl( T\Bigl(\phi[T_c/P](\num{x})\Bigr) \rightarrow T\Bigl(\phi[T_c/P](\num{Sx})\Bigr) \Bigr) \longrightarrow \\  \Bigl( T \Bigl(\phi[T_c/P](\num{0})\Bigr)  \rightarrow \forall x  \  T \Bigl(\phi[T_c/P](\num{x})\Bigr) \biggr).
\end{multline*}

If the \eqref{GC} principle were really the case, we could conclude that:

\begin{displaymath}
\Bigl(\forall x \  \bigl(\phi[T'_c/P](x) \rightarrow \phi[T'_c/P](Sx) \bigr)\Bigr) \longrightarrow \Bigl(  \phi[T'_c/P] (0) \rightarrow \forall x  \  \phi[T'_c/P](x) \Bigr).
\end{displaymath}  

Unfortunately, \eqref{GC} is strictly speaking not true, since we \emph{do not have} for example:

\begin{displaymath}
T(\exists x \ T_c(x)) \equiv \exists x \ T(T_c(x)),  
\end{displaymath}

but rather 

\begin{displaymath}
T(\exists x \ T_c(x)) \equiv \exists x \ T(T_c(\num{x})).
\end{displaymath}

Most probably, reformulating \eqref{GC} in a proper way would be quite messy. Instead of it, we use a different statement suggested to us by Cezary Cieśliński as a way to bypass the mentioned difficulty, which will appear as Lemma \ref{lem_generalised_commutativity}.

So we may hope to construct a family of predicates $T'_c(\num{\phi})$ which are compositional for formulae $\phi$ with the complexity smaller than $c$ (under an appropriate choice of the complexity measure) and satisfy the full induction scheme. Now we are very close to construct a compositional truth predicate satisfying the global reflection scheme. Namely, pick any model $(M,T)$ of $\CT_0$ and choose some formula $\psi$ and a set of sentences $\Gamma$ both in the domain of $M$ such that $M \models \Pr_{\Gamma}(\psi)$ with a proof $d$ in the domain of $M$. Now, one can check that the standard argument using induction on the structure of the proof  allows us to show that if all the sentences $\gamma$ in $\Gamma$ satisfy $T'_d(\gamma)$, then also $T'_d(\psi)$ holds.

The last step of our construction is simply to take the sum over $c \in M$ of the predicates $T'_c(\num{x}).$ If we can show that the predicates $T'_c(\num{x})$ and $T'_d(\num{x})$ agree on the formulae of complexity smaller than both $c$ and $d$, then the predicate $T'(x) = \bigcup_{c \in M} T'_c(x)$ should be fully compositional and satisfy the global reflection principle. Moreover, $T'(x)$ should satisfy $\Delta_0$-induction, since for any $a$ its restriction to the interval $[0,a]$ is equal to $T'_a \cap [0,a]$ and therefore fully inductive.

Now, our task is essentially twofold: first, we have to define the family of arithmetical predicates $(T_c(x))$ such that provably in $\CT_0$ the predicate $T_c(x)$ is compositional for formulae of complexity smaller than $c.$ Second, we have to spell out all the details of the above argument. The rest of this section will be devoted to these issues. Let us start with some definitions, which will play a crucial role in constructing our family of arithmetical truth predicates with nice properties.

\begin{definicja}
	Let $M\models \PA$, $c\in M$ and $(\phi_i)_{i\leq c}\in M$ be a coded sequence of sentences. Then $\bigvee_{i \leq c} \phi_i$ denotes the code of the disjunction of all $\phi_i$-s with parentheses grouped to the left (for the sake of determinateness only --- provably in $\CT_0$ the truth of a sentence does not depend on how we parenthesize blocks of disjunctions and conjunctions).
\end{definicja} 

\begin{lemat}[Disjunctive correctness] \label{disjunctive}
	$\CT_0$ proves that for all $x$ and all indexed families of sentences\footnote{Note that the quantification over indexed families of sentences is expressed here with an arithmetical formula ''for all $y$, if $y$ is a (code of a) sequence of arithmetical sentences of length $x$''. In particular, such a sequence will always have, according to $M$, the number of its elements equal to some $x \in M$. Therefore, from the point of view of $M$ it will be finite (but not necessarily equal to some $k \in \omega$).} $(\phi_i)_{i \leq x}$ 
	
	\begin{equation}\label{DC}\tag{DC}
	T\biggl(\bigvee_{i \leq x} \phi_i\biggr) \equiv \exists i \leq x \  \ T (\phi_i). 	
	\end{equation}			
	
\end{lemat}	

\begin{lemat}[The internal induction] \label{internal}
	$\CT_0$ proves the internal induction axiom \eqref{II}.
\end{lemat}

Proofs of both lemmata are carried out by a completely straightforward application of $\Delta_0$-induction. Let us show as an example the proof of disjunctive correctness.

\begin{proof}[Proof of Lemma \ref{disjunctive}]
	We show by induction on $x$ that the disjunctive correctness holds for all indexed families of sentences $a <x.$ In particular we will assume that under our coding, every subfamily of $a$ is also smaller than $x.$ We can assume that by convention a disjunction over empty sequence is some fixed \textit{falsum}, say $0 \neq 0$ and a disjunction $\bigvee_{i \leq 0} \phi_i$ over a sequence of length one is simply $\phi_0.$  
	
	Assuming the above conventions, the claim is trivial, when $a$ is the empty sequence or when $a$ has exactly one element. Suppose now that $a= \bigvee_{i \leq l} \phi_i$ and that $a<x+1.$ If $a<x$, then our claim holds by the induction hypothesis, so we may assume that $a=x.$ We will focus on the left-to-right direction in our equivalence. Let
	
	\begin{displaymath}
	a = \bigvee_{i \leq l} \phi_i.
	\end{displaymath} 
	
	Without loss of generality we may assume that $l>0.$ Then by definition we have:
	
	\begin{displaymath}
	T \Bigl( \bigvee_{i \leq l} \phi_i \Bigr) \equiv T \Bigl( (\bigvee_{i \leq l -1} \phi_i) \vee \phi_{l} \Bigr).
	\end{displaymath}
	
	By compositionality this implies:
	
	\begin{displaymath}
	T \Bigl( \bigvee_{i \leq l} \phi_i \Bigr) \equiv T \Bigl( \bigvee_{i \leq l -1} \phi_i \Bigr) \vee T(\phi_{l}).
	\end{displaymath}
	
	Which by induction hypothesis is equivalent to:
	
	\begin{displaymath}
	\Bigl( \exists i \leq (l-1) T(\phi_i) \Bigr) \vee T(\phi_l).
	\end{displaymath}
	
	The claim follows. The right-to-left direction is proved in a similar fashion.

\end{proof}

Before we proceed to the proof of our theorem, let us discuss a few preparatory steps. 

Let us introduce our main technical tool --- a particular class of partial truth predicates. Let $A_n$ be the set of arithmetical sentences whose
logical complexity is exactly $n$, where by the logical complexity we mean the maximal depth of nesting of logical symbols, i.e. quantifiers and connectives and each symbol is counted separately (e.g. prefixing a formula with a block of five universal quantifiers, raises its complexity by five). The binary relation $x \in A_y$ is clearly primitive recursive and thus represented in $\PA$. Its precise definition is as follows:
\begin{eqnarray*}
	x \in A_0 & \equiv & \exists s,t \ \ x= (s=t) \\
	x \in A_{y+1} & \equiv & \exists v, \phi(v) \ \ x = (\exists v \ \phi(v)) \wedge (\phi(0)) \in A_y \\
	& \vee & \exists v, \phi(v) \ \ x = (\forall v \ \phi(v)) \wedge   (\phi(0)) \in A_y \\
	& \vee & \exists \phi, \psi \ \ x = (\phi \vee \psi) \wedge \bigvee_{v,w: \max(v,w) = y}  \Bigl( \phi \in A_v \wedge \psi \in A_w \Bigr) \\
	& \vee & \exists \phi, \psi \ \ x = (\phi \wedge \psi) \wedge \bigvee_{v,w: \max(v,w) = y} \Bigl( \phi \in A_v \wedge \psi \in A_w \Bigr)\\
	& \vee & \exists \phi \ \ x =(\neg \phi) \wedge (\phi \in A_y). 
\end{eqnarray*}

Let us define a family of arithmetical predicates  $(\Theta_n)_{n\in\omega}$ in the following way:
\begin{eqnarray*}
	\Theta_0(x) & = & \exists s,t \ \ x = (s=t) \wedge s^{\circ} = t^{\circ}\\
	\Theta_{n+1}(x) & = & \exists v, \phi(v) \ \  x = (\exists v \ \phi(v)) \wedge \exists y \ \Theta_n(\phi(\num{y})) \\
	& \vee &  \exists v, \phi(v) \ \ x = (\forall v \ \phi(v)) \wedge \forall y \ \Theta_n(\phi(\num{y})) \\
	& \vee & \exists \phi, \psi \ \ x = (\phi \vee \psi) \wedge \bigvee_{k,l \leq n} \Bigl( \phi \in A_k \wedge \psi \in A_l \wedge (\Theta_k(\phi) \vee \Theta_l(\psi)) \Bigr) \\
	& \vee & \exists \phi, \psi \ \ x = (\phi \wedge \psi) \wedge \bigvee_{k,l \leq n} \Bigl( \phi \in A_k \wedge \psi \in A_l \wedge (\Theta_k(\phi) \wedge \Theta_l(\psi)) \Bigr)\\
	& \vee & \exists \phi \ \ x = (\neg \phi) \wedge \neg \Theta_n(\phi). 	 
\end{eqnarray*} 

Clearly, the functions $n \mapsto \qcr{A_n}$ and $n \mapsto \qcr{\Theta_n}$ are primitive recursive. Following our conventions we will write $A_x $ and $\Theta_x$ for the arithmetical formulae representing these functions as well as for their values. We will sometimes write $a \in A_c$ meaning that an element $a$ satisfies the formula $A_c(x)$ (where $c$ is a parameter, possibly nonstandard). Note, that  $A_c(x)$ may indeed be expressed with an arithmetical formula with a parameter.

Let a simplistic partial arithmetical truth predicate $T_n(x)$ be defined in the following way: 

\begin{displaymath}
T_n(x) = \bigvee_{j \leq n} x \in A_j \wedge \Theta_j(x).
\end{displaymath}

As before, we shall write simply $T_x(y)$ to denote the function $x,y \mapsto \qcr{T_x(y)}.$ Note that the definition of the predicates $T_n$ closely parallels that of the arithmetical satisfaction predicates for $\Sigma_n$-classes, only it is much simpler. Namely: we assume that \emph{every single} quantifier or connective increases the complexity of a formula and do not make a distinction between bounded and unbounded quantifiers.    

The key fact needed in the proof of our theorem is that if $T(x)$ satisfies the axioms of $\CT_0$, then partial truth predicates  defined for a parameter $c$ as 
\begin{displaymath}
T'_c(x) = T(T_c(\num{x})),
\end{displaymath}

or, in more detail, as
\begin{displaymath}
T'_c(x) = \exists y,z,w \ \Bigl( y = \num{x} \wedge z= \qcr{T_c(v)} \wedge w = \Subst(z,y) \wedge T(w) \Bigr) 
\end{displaymath}
enjoy remarkably good properties: they are compositional for formulae in the respective  classes $A_c$ and they are fully inductive.

\begin{lemat}[Compositionality of $T'_c$] \label{kompozycyjnosc ttc}
	$\CT_0$ proves that for every $y$ the following conditions hold:
	\begin{enumerate}
		\item $\forall s,t \ \ T'_y(s=t) \equiv s^{\circ}=t^{\circ}.$
		\item $\forall \phi,\psi   \ \ \biggl((\phi \otimes \psi) \in A_z \wedge y>z \rightarrow \Bigl(  T'_y(\phi \otimes \psi) \equiv T'_y \phi \otimes T'_y \psi \Bigr)\biggr).$
		\item $\forall \phi  \ \ \biggl((\neg \phi) \in A_z\wedge y>z \rightarrow \Bigl(T'_y(\neg \phi) \equiv \neg T'_y \phi \Bigr)\biggr).$
		\item $\forall v,\phi \ \ \biggl(( Q v \ \phi) \in A_z \wedge y>z \rightarrow \Bigl(T'_y(Q v \ \phi) \equiv Q x \ T'_y(\phi(\num{x})) \Bigr)\biggr).$ 	
	\end{enumerate}				
	
	Where $\otimes \in \{ \wedge, \vee \}$ and $Q \in \{\forall, \exists \}.$
\end{lemat}

\begin{proof}
	Let us fix arbitrary $c.$ We will prove that the above conditions hold for $T'_c$ and we will focus on the case of existential quantifier. Let $(\exists v \ \ \phi) \in A_j$ for some $j<c$ and suppose that 
	\begin{displaymath}
	T \Bigl(T_c(\num{\exists v \ \ \phi}) \Bigr).
	\end{displaymath}
	
	Then, by the disjunctive correctness of $T$, we must have $T(\eta)$ for some of the disjuncts $\eta$ in $T_c(\num{\exists v \ \phi}).$ Fix any $y$ and note that, since $(\exists v \ \ \phi) \in A_y$ is an arithmetical formula of standard syntactic structure with (possibly nonstandard) parameters, we have:
	
	\begin{displaymath}
	T\Bigl( (\exists v \  \phi) \in A_{\num{y}} \Bigr) \equiv (y=j).
	\end{displaymath}
	
	Therefore the only possible candidate for our true disjunct $\eta$ is:
	
	\begin{displaymath}
	T\Bigl( (\exists v \  \phi) \in A_{\num{j}} \wedge \Theta_j(\num{\exists v \ \ \phi}) \Bigr).
	\end{displaymath}
	
	This implies that in particular we have:
	
	\begin{displaymath}
	T\Bigl(\Theta_j(\num{\exists v \ \phi})\Bigr).
	\end{displaymath} 
	
	Which entails:
	
	\begin{displaymath}
	T \Bigl( \exists y \ \ \Theta_{j-1} (\num{\phi(\num{y})}) \Bigr).
	\end{displaymath}

	This by definition equals to the following statement:
	
	\begin{displaymath}
	T \Bigl(  \exists y, z, \psi ( \ \ z = \num{y} \wedge \psi = \Subst(\num{\phi(v)}, z) \wedge \Theta_{j-1}(\psi)) \Bigr).
	\end{displaymath}
	
	Since the compositional truth predicate $T$ satisfies Tarski biconditionals for standard sentences with possibly nonstandard numerals, this implies:
	
	\begin{displaymath}
	\exists y, z, \psi \ \Bigl( (z = \num{y}) \wedge (\psi = \Subst(\phi(v), z)) \wedge T(\Theta_{j-1}(\num{\psi})) \Bigr),
	\end{displaymath}
	
	which may be again abbreviated as: 
	
	\begin{displaymath}
	\exists x \ \ T \Bigl( \Theta_{j-1}(\num{\phi(\num{x})}) \Bigr).
	\end{displaymath}
	
	This, again by the disjunctive correctness property, entails:
	
	\begin{displaymath}
	\exists x \ \ T\Bigl( T_c(\num{\phi(\num{x})}) \Bigr).
	\end{displaymath}
	
	Finally, by definition this is simply:
	
	\begin{displaymath}
	\exists x \ \ T'_c (\phi(\num{x}))
	\end{displaymath}
	
	The other cases are analogous.
	
\end{proof}

The next lemma we will use is the key ingredient of our non-conservativeness result. It states that for arbitrary $c$ (possibly nonstandard) the predicates $T'_c(x)$ satisfy the full induction scheme.

\begin{lemat} \label{lemat_induktywnosc_TTc}
	Let $\phi(v)$ be any formula of the arithmetical language expanded with a fresh unary predicate $P(v)$. Then $\CT_0$ proves that for every $c$
	\begin{displaymath}
	\Bigl(\forall x \  \bigl(\phi[T'_c/P](x)  \rightarrow \phi[T'_c/P](Sx) \bigr)\Bigr) \longrightarrow \Bigl(  \phi[T'_c/P] (0) \rightarrow \forall x  \  \phi[T'_c/P](x) \Bigr).
	\end{displaymath}  
\end{lemat}

Note that in the above lemma by $\phi[T'_c/P](x)$ we mean an actual formula, rather than its formalised version. Thus, in effect, the above lemma states that the formulae $T'_c$ are really inductive for an arbitrary choice of the parameter $c.$ 

As outlined at the beginning of the current section, the above lemma would be very easily proved, if the generalized commutativity principle \eqref{GC} were true, which is unfortunately not the case. We shall bypass the difficulty with the following lemma.\footnote{We are grateful to Cezary Cieśliński for formulating this fact and suggesting it as a way to prove inductiveness of $T'_c$ in a proper manner.}

\begin{lemat} \label{lem_generalised_commutativity}
	Let $\phi(x_1, \ldots, x_n)$ be an arbitrary formula, in the arithmetical language augmented with a fresh predicate $P(x).$ Then $\CT_0$ proves that for every $c$ there exists an arithmetical formula $\psi(x_1, \ldots, x_n)$ such that
	\begin{displaymath}
	\forall x_1, \ldots, x_n \ \biggl( \phi[T_c'/P](x_1, \ldots, x_n) \equiv T(\psi(\num{x_1}, \ldots, \num{x_n})) \biggr).
	\end{displaymath}
\end{lemat}

\begin{proof}
	
	We proceed by (meta-)induction on the complexity of a formula $\phi.$ If $\phi$ is an arithmetical atomic formula $s(x_1, \ldots, x_n) = t(x_1, \ldots, x_n)$ then:
	
	\begin{displaymath}
	\phi[T_c'/P](x_1, \ldots, x_n)  = \phi(x_1, \ldots, x_n) \equiv T(\phi(\num{x_1}, \ldots, \num{x_n})).
	\end{displaymath}
	
	If $\phi(x) = P(t(x_1, \ldots, x_n))$, then:
	
	\begin{displaymath}
	\phi[T_c'/P](x_1, \ldots, x_n) = T'_c(t(x_1, \ldots, x_n)).
	\end{displaymath}
	
	By definition the last formula may be expanded to:
	
	\begin{displaymath}
	\exists x,y,z \ \ \Bigl( x = t(x_1, \ldots, x_n) \wedge y= \num{x} \wedge z = \Subst(T_c,y) \wedge T(z) \Bigr).
	\end{displaymath}
	
	Let 
	
	\begin{displaymath}
	\psi(x_1, \ldots, x_n) = \exists x \ \Big( x= t(x_1, \ldots, x_n) \wedge T_c(x) \Big).
	\end{displaymath}
	
	Now, by compositional clauses in $\CT^-$ we have the following equivalences:

	\begin{eqnarray*}
		T \bigg( \exists x \ \Big( x= t(\num{x_1}, \ldots, \num{x_n}) \wedge T_c(x) \Big) \bigg) & \equiv & \exists x \ T \Big( \num{x}= t(\num{x_1}, \ldots, \num{x_n}) \wedge T_c(\num{x}) \Big) \\
		& \equiv & \exists x \ \Big( \num{x}^{\circ} = t(\num{x_1}, \ldots, \num{x_n})^{\circ} \wedge T(T_c(\num{x}))  \Big) \\
		& \equiv & \exists x \ \Big( x= t(x_1,\ldots, x_n) \wedge T(T_c(\num{x})) \Big).  
	\end{eqnarray*}
	
	Note, that the term $t$ above is standard and, consequently, it may be written down explicitly in the last step. Now, the last formula in the above equivalences is precisely the abbreviation of the following one:
	
	\begin{displaymath}
	\exists x,y,z \ \ \Bigl( x = t(x_1, \ldots, x_n) \wedge y= \num{x} \wedge z = \Subst(T_c,y) \wedge T(z) \Bigr).
	\end{displaymath}
	
	So the claim follows with $\psi$ as above.

	If $\phi$ is a boolean combination of some formulae $\xi, \eta,$ then the induction step is straightforward, so suppose that $\phi = \exists y \ \ \eta(x,y).$ Then we have:
	
	\begin{displaymath}
	\phi[T_c'/P](x_1, \ldots, x_n) = \exists \ y (\eta[T_c'/P](x_1, \ldots, x_n ,y)).
	\end{displaymath}
	
	By induction hypothesis, there exists $\psi'(x_1, \ldots, x_n,y)$ such that
	
	\begin{displaymath}
	\forall x_1, \ldots, x_n, y \ \Bigl(\eta[T'_c/P](x_1, \ldots, x_n, y) \equiv T(\psi'(\num{x_1}, \ldots, \num{x_n}, \num{y}))\Bigr). 
	\end{displaymath}
	
	Using the compositional clauses for $T$, we have:
	
	\begin{displaymath}
	\forall x_1, \ldots, x_n \Bigl(\phi[T_c'/P](x_1, \ldots, x_n) \equiv T(\exists y \psi'(\num{x_1}, \ldots, \num{x_n},y)) \Bigr).
	\end{displaymath}
	
	So, the claim holds with $\psi = \exists y \psi'.$ The universal quantifier case is analogous.

\end{proof}

\begin{proof}[Proof of Lemma \ref{lemat_induktywnosc_TTc}]
	Fix $\phi(v)$ as in the claim of the lemma. Working in $\CT_0$, fix an arbitrary $c.$ By the previous lemma, there exists an arithmetical formula $\psi(x)$ such that
	
	\begin{displaymath}
	\forall x \ \Bigl( \phi[T'_c/P](x) \equiv T\psi(\num{x}) \Bigr).
	\end{displaymath}
	
	By the internal induction principle we have:
	
	\begin{displaymath}
	\Bigl(\forall x \  \bigl(T(\psi(\num{x})) \rightarrow T(\psi(\num{Sx})) \bigr)\Bigr) \longrightarrow \Bigl(  T(\psi(\num{0}) \rightarrow \forall x  \  T(\psi(\num{x})) \Bigr).
	\end{displaymath}  
	
	By lemma \ref{lem_generalised_commutativity} this entails:
	
	\begin{displaymath}
	\Bigl(\forall x \ \bigl(\phi[T'_c/P](x)  \rightarrow \phi[T'_c/P](Sx) \bigr)\Bigr) \longrightarrow \Bigl(  \phi[T'_c/P] (0) \rightarrow \forall x  \  \phi[T'_c/P](x) \Bigr).
	\end{displaymath}
	
	Hence Lemma \ref{lemat_induktywnosc_TTc} follows.  
\end{proof}

Let us  stress that we did not put any restrictions on the complexity of $\phi,$ so that what we obtain for the predicates $T'_c$ is the full induction, rather than $\Delta_0$-induction as one could possibly expect.

Now let us state another important lemma. It states that the predicates of the form $T'_c$ are compatible.

\begin{lemat} \label{lem: zgodnosc i okreslonosc TTc}  Provably in $\CT_0$ the predicates $T'_d$ have the following properties:
\begin{enumerate}
	\item If $x \notin A_c$ for all $c \leq d$, then  we have $\neg T'_d(x).$
	\item 	For arbitrary $d<e$ and $\phi \in A_d$ we have $T'_d(\phi) \equiv T'_e(\phi).$
\end{enumerate} 
\end{lemat}

\begin{proof}
	Both claims are a  straightforward application of the disjunctive correctness and the fact that the formula ''$x \in A_y$'' is standard, so we have: 
	\begin{displaymath}
	\forall x,y  \Bigl(T(\num{x} \in A_{\num{y}}) \equiv (x \in A_y) \Bigr).
	\end{displaymath}
\end{proof}

Now we introduce one more truth predicate, whose properties will almost immediately imply our theorem. 

\begin{definicja} \label{def_T'}
	Let the formula $T'(x)$ be defined in the following way:
	\begin{displaymath}
	T'(x) = \exists v \  \bigl( (x \in A_v) \wedge TT_v(\num{x}) \bigr).
	\end{displaymath}
\end{definicja}

Intuitively, the extension of the predicate $T'(x)$ is the sum of the extensions of predicates $T'_c(x)$. We should expect that it behaves reasonably,  since by Lemma \ref{lem: zgodnosc i okreslonosc TTc} the predicate $T'_c(\phi)$ is an extension $T'_d(\phi)$ to arithmetical sentences in $A_c \setminus A_d$ whenever $c>d$. Provably in $\CT_0$ the newly introduced predicate $T'(x)$ satisfies  the following four properties:

\begin{enumerate}
	\item $\Delta_0$-induction scheme.
	\item Compositionality.
	\item Regularity.
	\item Global reflection principle.
	\item Axiom soundness property for $\PA.$
\end{enumerate}

Now we will spell out the listed properties in the series of lemmata. 

\begin{lemat}[Bounded induction for $T'$] \label{lem_bounded_induction_T'}
	For any $\Delta_0$ formula $\phi$ in the arithmetical language enriched with the fresh unary predicate $P(x)$, $\CT_0$ proves that:
	\begin{displaymath}
	\forall x \  \Bigl(\phi[T'/P](x) \rightarrow \phi[T'/P](Sx) \Bigr) \longrightarrow \Bigl(  \phi[T'/P] (0) \rightarrow \forall x  \  \phi[T'/P](x) \Bigr).
	\end{displaymath}   
\end{lemat}

In the proof of the above lemma we will use the following characterisation of $\Delta_0$-induction:\footnote{See \cite{kossakschmerl}, Proposition 1.4.2.}

\begin{fakt}
	Let $A(x)$ be an arbitrary formula in a language $\mathcal{L}$ extending the language of the arithmetic. Then the following conditions are equivalent for an arbitrary $\mathcal{L}$-structure $M$ satisfying $\PA$:
	\begin{enumerate}
		\item $M \models$  $\Delta_0$-induction for the formula $A(x)$.
		\item For every $b$ there exists an element (a coded set) $c$ such that 
		\begin{displaymath}
		M \models \forall x < b \ \ \biggl( A(x) \equiv x \in c\biggr),
		\end{displaymath}
		in which case we say that the set of elements below $b$ satisfying $A(x)$ in $M$ is coded. 
	\end{enumerate}
\end{fakt}

\begin{proof}[Proof of Lemma \ref{lem_bounded_induction_T'}]
	
	We will show that for an arbitrary model $M$ of $\CT_0$ and arbitrary element $b \in M$ the set of elements below $b$ satisfying $T'(x)$ is coded in $M$. Then the claim of the lemma will follow by the above Fact.
	
	In any model of $\CT_0$ the extension of the predicate $T'$ is the sum over $a$'s of the extensions of the predicates $T'_a$. Now observe that for arbitrary $b$ and arbitrary $x <b$ the following conditions are equivalent by Lemma \ref{lem: zgodnosc i okreslonosc TTc} (if we assume, as we may, that the complexity of every formula is no greater than its code, i.e. for arbitrary $x$ we have $x \in A_y$ for some $y \leq x$):
	\begin{enumerate}
		\item $T'(x).$
		\item $T'_b(x).$
	\end{enumerate}

	Now, there clearly exists an element $c$ such that:
	\begin{displaymath}
	\forall x < b \ \ \biggl(T'_b(x) \equiv x \in c\biggr),
	\end{displaymath}
	since the predicate $T'_b(x)$ is fully inductive by Lemma \ref{lemat_induktywnosc_TTc}.
\end{proof}

\begin{lemat}[Compositionality of $T'$] \label{lem_compositionality_T'}
	Provably in $\CT_0$ the following conditions hold:
	\begin{enumerate}
		\item $\forall s,t \ \Bigl( T'(s=t) \equiv s^{\circ}=t^{\circ} \Bigr).$
		\item $\forall \phi,\psi   \ \Bigl(  T'(\phi \otimes \psi) \equiv T'\phi \otimes T'\psi \Bigr) .$
		\item $\forall \phi  \ \Bigl( T'(\neg \phi) \equiv \neg T'\phi \Bigr) .$
		\item $\forall v,\phi(v) \ \Bigl( T'(Q v \ \phi(v)) \equiv Q x \ T'(\phi(\num{x})) \Bigr).$ 	
	\end{enumerate}				
	
	Where $\otimes \in \{ \wedge, \vee \}$ and $Q \in \{\forall, \exists \}.$
\end{lemat}

\begin{proof}
	The lemma follows immediately from the definition of $T'(x)$, Lemma \ref{lem: zgodnosc i okreslonosc TTc} and Lemma \ref{kompozycyjnosc ttc}.
\end{proof}

The fact that the predicate $T'(x)$ may be presented as a sum of the predicates $T'_c$, which are fully inductive and compositional for formulae of logical complexity no greater than $c$ guarantees that $T'$ enjoys all sorts of good properties. The next lemma states that it is fully extensional.

\begin{lemat}[Regularity of $T'$] \label{lem: regularnosc T prim}
	The predicate $T'$ satisfies the regularity axiom, i.e.
	\begin{displaymath}
		\forall \phi \forall t,s\ \ \bigl(  s^{\circ} = t^{\circ} \rightarrow  T'\phi(t)\equiv T'\phi(s)\bigr). 
	\end{displaymath}
	The same hold for predicates $T'_b$ for arbitrary $b$ 
\end{lemat}
\begin{proof}
	 We know that for all $d \leq b$ and $x \in A_d$ we have 
	\begin{displaymath}
	T'(x) \equiv T'_b(x).
	\end{displaymath}
	Thus it is enough to prove that for arbitrary $b$ the following equivalence holds:
		\begin{displaymath}
		\forall d \leq b \forall \phi(v) \in A_d \forall t,s\ \ \bigl(  s^{\circ} = t^{\circ} \rightarrow  T'_b\phi(t)\equiv T'_b\phi(s)\bigr). 
		\end{displaymath}
	This however may be shown by a straightforward $\Pi_1$-induction on $d.$
\end{proof}

\begin{lemat}[Quantifier axioms for $T'$] \label{lem: aksjomaty kwantyfikatorowe dla T prim}
	Provably in $\CT_0$ The predicate $T'$ satisfies the compositional axioms for quantifiers in the term formulation, i.e.
	\begin{enumerate}
		\item $\forall v, \phi(v) \ \ T'(\forall v \ \phi(v)) \equiv \forall t \ T'(\phi(t)).$ 
		\item  $\forall v, \phi(v) \ \ T'(\exists v \ \phi(v)) \equiv \exists t \ T'(\phi(t)).$
	\end{enumerate}
\end{lemat}

\begin{proof}
	This follows immediately from Lemma \ref{lem: regularnosc T prim}, since provably in $\PA$ (and in fact in much weaker theories) for every term $t$ there exists a unique numeral $\num{a}$ with $t^{\circ} = (\num{a})^{\circ}.$   
\end{proof}

Recall that by the global reflection principle for $T'$ we mean the axiom:
\begin{displaymath}
\forall \phi \ \Bigl(\Prov_{T'} \phi \longrightarrow T' \phi \Bigr),
\end{displaymath}
where $\Prov_{T'}$ is an instance of $\Prov_{\tau}$ defined in Convention \ref{konw: glupia konwencja} with $\tau = T'.$ Intuitively $\Prov_{T'}$ means first-order provability from some premises $\phi$ satisfying $T'(\phi).$

\begin{lemat}[Global reflection for $T'$] \label{lem_global_reflection_T'}
	$\CT_0$ proves the global reflection principle for the formula $T'.$
\end{lemat}

\begin{proof}
	Working in $\CT_0$, take an arbitrary (code of a) proof $d$ with the conclusion $\phi,$  all of whose premises $\psi$ satisfy $T'(\psi).$ Then there exists a $b$ such that for every formula $\eta$ in $d$ there exists $y<b$ such that:
	\begin{displaymath}
	\eta \in A_y.
	\end{displaymath}
	
	Then for each $x<b$ we have:
	
	\begin{displaymath}
	T'(x) \equiv T'_b(x).
	\end{displaymath}
	
	So it is enough to prove that $T'_b(\phi)$ holds. But this may be easily proved by Lemmata \ref{lemat_induktywnosc_TTc} and \ref{lem: aksjomaty kwantyfikatorowe dla T prim} analogously to the usual proof of the global reflection principle for $\CT$ with compositional clauses for quantifiers in term formulation and with the full induction for the compositional truth predicate, since the formulae $T'_b(x)$ are fully inductive by Lemma \ref{lemat_induktywnosc_TTc} and compositional for formulae of logical complexity at most $b$ by Lemma \ref{kompozycyjnosc ttc}. They satisfy compositional axioms for quantifiers in the term formulation, since $T'$ does by Lemma \ref{lem: aksjomaty kwantyfikatorowe dla T prim} and as already noted for all $x<b$ we have 	$T'(x) \equiv T'_b(x)$.  
\end{proof}

Analogously, we obtain the following lemma:

\begin{lemat}[Axiom soundness property] \label{lem_surface_soundness_PA}
	$\CT_0$ proves the following sentence:
	\begin{displaymath}
	\forall \phi \ \Bigl( \Ax_{\PA}(\phi) \longrightarrow T' (\phi) \Bigr).
	\end{displaymath}
\end{lemat}

Now we are ready to prove the main theorem of our paper.

\begin{proof}[Proof of Theorem \ref{tw_niekonserwatywnosc_ct0}]
	
	We want show that $\CT_0$ $\PA$-relatively defines $\CT_0$ extended with the global reflection principle and the axiom soundness property for $\PA$. It is enough to observe that by the lemmata \ref{lem_bounded_induction_T'}, \ref{lem_compositionality_T'}, \ref{lem_global_reflection_T'} and \ref{lem_surface_soundness_PA} the formula $T'(x)$ satisfies the axioms of the latter theory, so it gives us the interpretation of the latter theory in $\CT_0$ which fixes the arithmetical vocabulary.
\end{proof}

Let us observe that to establish the non-conservativity result for $\CT_0$ one actually needs only two properties of the compositional truth predicate: disjunctive correctness and internal induction. By Theorem \ref{EV} of Enayat--Visser and Leigh\footnote{See \cite{enayatvisser}, remarks in Section 6 and \cite{leigh}, Theorem 3.} the theory $\CT^-$ augmented with the latter principle is still conservative. We do not know yet whether this is true also for the former. 

Let us end this section by a corollary we have referred to in the abstract. Recall the usual version of $\CT^-$ with compositional axioms for quantifiers of the form
\begin{displaymath}
\forall \phi(v) \ \ T(Q v \ \phi(v)) \equiv Q t \ T(\phi(t)).
\end{displaymath}

Let us denote by $\CT^-_t$ this usual version of $\CT^-$. Since the regularity principle implies that the theories $\CT^-$ and $\CT^-_t$ are equivalent, our results actually imply:

\begin{wniosek} \label{wn_zwykle_ct-}
	$\CT^-_t$ extended with $\Delta_0$-induction for the full language and the regularity principle is not conservative over $\PA.$
\end{wniosek}

		\section{A variation of the main result}
		
		In this section we shall prove that $\CT_0$ is actually very close to proving the global reflection principle for its truth predicate. We will show that a very natural modification of this theory accomplishes this aim. First we shall formulate $\CT_0$ over a theory $\PA^+$ which is simply $\PA$ formalized in an enriched language with additional function symbols for some primitive recursive functions and extended with axioms determining the meaning of new symbols. Observe that the axioms of $\CT_0$ remain unchanged, but the notion of \textit{a term} is substantially enriched. To such a theory we add an axiom which generalizes the regularity principle \eqref{regularity} to substitutions of boundedly-many terms at once.
		

		Now we introduce a short sequence of definitions:
		
		\begin{definicja}\
			\begin{enumerate}
				\item Let $\mathcal{L}^+ = \mathcal{L}_{\PA}\cup\{(x)_y\}$, where $(x)_y$ is a two-argument function (with the intended reading  ``the result of the projection of $x$ to its $y$-th component'').
				\item Let $\pi(x,y,z)$ be an $\mathcal{L}_{\PA}$ formula expressing that $z$ is the $y$-th element of the sequence $x$. \textbf{$\PA^+$} is the theory theory in $\mathcal{L}^+$ containing the usual axioms of $\PA$ (we allow formulae of $\mathcal{L}^+$ in the induction axioms) and additionally the following axiom for $(x)_y$ 
				\[\forall x,y,z\ \ \bigr((x)_y = z \equiv \pi(x,y,z)\bigl)\]
			\end{enumerate}
		\end{definicja}
		
		\begin{konwencja}
			For each $n$, $v=\Seq{x_1,\ldots,x_n}$ is the arithmetical formula representing in $\PA$ the relation ''$v$ is the code of the sequence of length $n$, containing as its first element $x_1$, as its second $x_2$,..., and as its $n$-th $x_n$''.
		\end{konwencja}
		
		We extend the arithmetization so that it embraces the two-argument symbol $(x)_y$. We assume that in $\PA^+$ the definable syntactical relations apply to the enriched language, hence $\textnormal{Term}(x)$, $\textnormal{Form}(x)$ mean ''$x$ is a (code of a) $\mathcal{L}^+$ term'', ''$x$ is a (code of a) $\mathcal{L}^+$ formula'' respectively. Additionally, we need the following formulae \footnote{Mind that \textit{sequence}, \textit{term} etc. in the definition below mean \textit{a code of sequence}, \textit{a code of term} rather than truly finite sequence, true term, defined externally.}:
		
		\begin{definicja}[$\PA^+$]\
			\begin{enumerate}
				\item $\TermSeq(x)$ which says ''$x$ is a sequence of closed terms.'' 
				\item We extend the function $\val{\cdot}$ to terms of $\mathcal{L}^+$ by putting
				\[\val{(t)_s} = (\val{t})_{\val{s}}\] 
				where $t$, $s$ are arbitrary terms. In $\PA^+$ we define a generalized function of valuation which, apart from terms, is applicable also to \emph{sequences} of terms, yielding the sequence of values of those terms. Since we will be interested only in values of sequences of terms, we will denote this function by $\val{\cdot}$, in the same way in which the standard function of valuation was denoted.
				\item  $\Subst(x,y)$ is now a \emph{generalized} function of substitution with the following properties
				\begin{enumerate}
					\item $y$ must be a sequence of closed terms such that if $a$ is the number of distinct variables 
					occurring in $x$ then the length of $y$ is at least $a$,
					
					\item  $\Subst(x,\tau)$ is the effect of formal substitution in $x$ for free variable $x_z$ the term coded by the $z$-th element of sequence $\tau$ (for every $z$).

				\end{enumerate} 
				As indicated in the previous section, in order not to complicate the formulae, we shall be writing $\phi(\tau)$ instead of $\Subst(\phi,\tau)$, even in the case when $\tau$ is a sequence of terms.
				\item Let $\tau$ be any sequence. By $\tau^*$ we denote the unique sequence $y$ such that
				\begin{enumerate}
					\item $\len(\tau) = \len(y)$
					\item $y$ has on its $i-$th place ($i<\len(y) = \len(\tau)$) the G\"odel code of the term $(\num{\tau})_{\num{i}}$.\footnote{Formally: $\forall i< \len(y)\ \ (y)_i = \qcr{(}\conc \num{\tau}\conc\qcr{)}\conc \num{i}$.}
				\end{enumerate} 
				
			\end{enumerate}
			
		\end{definicja}
		
		\begin{przyklad}
			Let $n = \Seq{\qcr{SS(0)},\qcr{S(0)+0}, \qcr{S(S(0) +S(0))}}$. Then
			\[\val{n} = \Seq{2, 1, 3} \]
			and
			\[n^* = \Seq{\qcr{(S^n0)_{S0}}, \qcr{(S^n0)_{SS0}}, \qcr{(S^n0)_{SSS0}} }\]
			and $\val{\qcr{(S^n0)_{S(0) + S(0)}}} = (n)_{2} = \qcr{S(0)+0}$.
		\end{przyklad}
		
		\begin{uwaga}[$\PA^+$]\label{rem: t=valt*}
			Let $\tau$ be any sequence. Then $\tau^*$ is a sequence of closed terms and $\tau=\val{\tau^*}$.
		\end{uwaga}

		\begin{uwaga}
			Mind the difference between $\phi(\tau^*)$ and $\phi(\tau)$. For example if 
			$$\phi = \qcr{x_1+x_2 = x_3}$$ 
			and $\tau = \Seq{\qcr{0},\qcr{S0+SS0},\qcr{SS0}}$, then (provably in $\PA^+$)
			\[\phi(\tau)=_{def}\Subst(\phi,\tau) = \qcr{0+(S0+SS0)=SS0}\]
			and 
			\[\phi(\tau^*)=_{def}\Subst(\phi,\tau^*) = \qcr{(\num{\tau})_{S0}+(\num{\tau})_{SS0}=(\num{\tau})_{SSS0}}\]
		\end{uwaga}
		Note that in the previous sections it was irrelevant for our argumentation whether the above mentioned functions are primitive or defined symbols. Now it becomes crucial to our argument, and the reason for extending the language will become apparent when proving the main theorem of this section.
		
		\begin{definicja}
			Let $\phi$ be an $\mathcal{L}^+$ formula. We say that an occurrence of term $t$ in $\phi$ is \df{bounded}  if and only if it contains a bounded occurrence of a variable. An occurrence of a term $t$ is \df{free} if it is not bounded. 
		\end{definicja}
		
		\begin{przyklad}
			$\qcr{SS(v)}$ and $\qcr{SS(v) + S(y)}$ have bounded occurrences in
			\[\phi = \qcr{\exists v \ \ SS(v) + S(y) = SS(z)}\]
			but $\qcr{S(y)}$ and $\qcr{SS(z)}$ don't.
		\end{przyklad}
		
		We will need a more refined measure of complexity of formulae than the one defined in the previous section (called \textit{logical complexity} there). To define precisely the conditions which it should meet, let us introduce one more notion, borrowed from \cite{leigh}:
		\begin{definicja}[$\PA^+$]
			Let $\phi$ be an $\mathcal{L}^+$-formula and $w$ be any number which is not a code of any $\mathcal{L}^+$ symbol. Let $\mathcal{L}^w$ be a language resulting by adding $w$ to $\mathcal{L}^+$. We treat $w$ as an additional free variable for marking places for terms in a formula. Formally: the formula $\phi'$ results from $\phi$ by formally substituting $w$ for every free variable of $\phi$. The formula $\bar{\phi}$ results from $\phi'$ by formally substituting $w$ for every term $t$ such that every variable occurring in $t$ is equal to $w$. If $\psi$ is any $\mathcal{L}^+$- formula, we put
			\[\phi\sim\psi\]
			iff $\bar{\phi} = \bar{\psi}$.
		\end{definicja}  
		The above definition serves for formalizing the relation of ''being the same up to substitution of terms for free occurrences of terms''. We demand that our measure of complexity have two properties:\footnote{''FIN'' is a short for ''FINite'' and ''COM'' for ''COMpositional''.}
		\begin{enumerate}
			\item[FIN] Provably in $\PA^+$, for every $n,k$, there are only finitely many $\sim$-equivalence classes of formulae of complexity less than $n$ which use variables (either as bounded or free ones) with indices smaller than $k$.
			\item[COM] The measure of $\phi\otimes\psi$ and $\neg \psi$ is greater than the measure of $\phi,\psi$ (for $\otimes\in \{\wedge,\vee\}$). The measure of $Qx\phi(x)$ (for $Q\in \{\forall,\exists\}$) is greater than the measure of $\phi(t)$ for every closed term $t$.
		\end{enumerate}
		The measure given by the logical complexity of a formula, as used in the previous section, does not satisfy property FIN (at least for formulae of language $\mathcal{L}$ (and consequently $\mathcal{L}^+$) which contains infinitely many closed terms\footnote{For example there are infinitely many sentences of the form $t_1 = t_2$ and syntactic tree for every such formula has logical complexity $0$.}). If we tried to use the G\"{o}del number of $\phi$ as its size, then the resulting measure would not satisfy COM. The next definition supplies us with an appropriate measure. 
		
		\begin{definicja} Let $\phi$ be an $\mathcal{L}^w$ formula.
			\begin{enumerate}
				\item The \df{syntactic tree} of an $\mathcal{L}^w$ formula $\phi$ is defined as usual except for the fact that we unravel also terms occurring in $\phi$. In consequence the only $\mathcal{L}^w$ symbols that are allowed to occur in leaves of the syntactic tree of $\phi$ are individual constants and variables.
				
				\item We say that a formula $\phi$ has \df{complexity at most} $n$ if and only if the height of the syntactic tree of $\bar{\phi}$ (i.e. the largest number of vertices on a maximal path) is at most $n$. The \df{complexity} of $\phi$ is the least $n$ such that $\phi$ is of complexity at most $n$.
			\end{enumerate}	
		\end{definicja}
		
		It is straightforward to check that our definition of complexity measure satisfies FIN and COM as stated above.  $\bar{\phi}$ will serve us also to define the \emph{template} of $\phi$. In fact $\bar{\phi}$ is very close to match our objectives: the only improvement we need to introduce is to number the occurrences of $w$ in $\bar{\phi}$. The next definition accomplishes this aim: 
		
		\begin{definicja}[$\PA^+$]\label{def: dluga}\
			\begin{enumerate}
				\item Let $\was{e_i}$ be an injective enumeration of a set of numbers which are not codes of any $\mathcal{L}^+$ symbols (starting from $e_1$ for simplicity). We will treat them as additional free-variable symbols (we will allow substituting closed terms for them but they are not part of $\mathcal{L}^+$). Let $\mathcal{L}^*$ be the language resulting from $\mathcal{L}^+$ by adding $\was{e_i}$ as new variable symbols.\footnote{In particular we assume that provably in $\PA^+$ $\mathcal{L}^+$ is a sublanguage of $\mathcal{L}^*$.}
				
				\item Let $\phi$ be an $\mathcal{L}^w$ formula and $x_1,x_2$ be two occurrences of free variables in $\phi$. We put 
				\[x_1\preccurlyeq_{\phi} x_2\]
				if and only if $x_1$ is more to the left in the syntactic tree of $\phi$ than $x_2$.
				
				\item Suppose now $\psi$ is an $\mathcal{L}^+$ formula. $\phi^*$ is the formula resulting from $\bar{\phi}$ by substituting $e_0,e_1, e_2, \ldots$ in $\bar{\phi}$ for the first, the second, the third, $\ldots$ occurrence of $w$ respectively, 
				(where the respective ordering of occurrences of free variables is $\preccurlyeq_{\phi}$). $\phi^*$ is called \df{the template of} $\phi$. Let us note that the only variables which occurs freely in $\phi^*$ are the $e_i$'s. In particular no variables from $\mathcal{L}$ occurs freely in $\phi^*$. Let us note also that for every $\phi\in\mathcal{L}^+$
				\[\phi\sim\phi^*\]
				
				\item Provably in $\PA^+$, for every $\phi\in \mathcal{L}^+$ there exists a coded set of those indices $i$ such that $e_i$ occurs in $\phi^*$. For a given $\phi$, such a set will be denoted by $E(\phi)$.
				
				\item We define a natural extension of the function $\Subst$ so that it can operate on templates (and denote with the same symbol and use the same conventions): we assume that if $\phi^*$ is the template of an $\mathcal{L}^+$ formula and $\tau$ is a sequence of $\mathcal{L}^+$ closed terms, then $\Subst(\phi^*, \tau)$ is the result of the following substitution in $\phi^*$: for every $i\in E(\phi)$\footnote{For simplicity we do not assume that $\Subst$ is defined for all formulae of $\mathcal{L}^*$, but only for formulae of $\mathcal{L}^+$ and templates (which do not contain any free variables from $\mathcal{L}^+$.)}. 
				\[\textnormal{the $i$-th element of $\tau$ is substituted for $e_i$}\]
				
				\item\label{tefi} Provably in $\PA^+$ for every $\phi\in \mathcal{L}^+$ the template of $\phi$ is uniquely determined. Moreover there is only one sequence of terms $\tau$ such that 
				\begin{enumerate}
					\item $\phi = \phi^*(\tau)$
					\item $\len(\tau) = \max\set{i}{e_i\in E(\phi)}$
				\end{enumerate}
				The last condition is added only to guarantee uniqueness of such $\tau$, and it won't play any important role in our proof. Such a sequence of terms will be denoted by $\tau_{\phi}$.
			\end{enumerate}
		\end{definicja}

		\begin{przyklad}
			The syntactic tree of $\qcr{\exists v \ \ SS(v) + S(y) = SS(z)}$ is the (code of the) following:
			
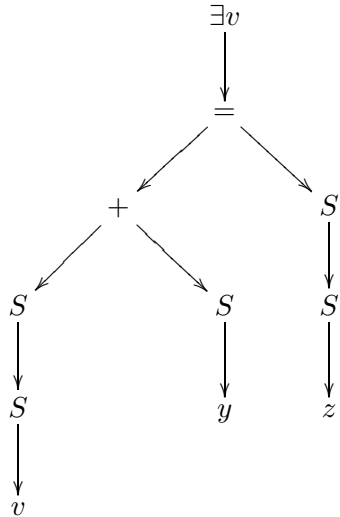
\begin{displaymath}
			\xymatrix{ & &\exists v\ar[d] & \\
				& & = \ar[dl]\ar[dr]&\\
				&+\ar[dl]\ar[dr]& & S\ar[d]\\
				S\ar[d]& & S\ar[d] & S\ar[d]\\
				S\ar[d]& & y & z \\
				v & & &}			
			\end{displaymath}
			The ordering on the set of occurrences of free variables in this formula is simply:
			\begin{equation*}
			\qcr{y} \preccurlyeq_{\phi} \qcr{z}
			\end{equation*} 
		\end{przyklad}
		\begin{przyklad}\label{przyk: compl, template}
			The following formula
			\[\phi = \qcr{\exists z\ \ SSS(z) + v = SSSSSSSS(y)}\]
			has complexity $7$. $\bar{\phi}$ is equal to $\qcr{\exists z\ \ SSS(z) + w = w}$. The template of $\phi$ is
			\[\phi^* = \qcr{\exists z\ \ SSS(z) + e_1 = e_2}\]
			and the corresponding sequence of terms is
			\[\tau_{\phi}=\Seq{\qcr{v},\qcr{SSSSSSSS(y)}}\]
			To give another example: if
			\[\psi = \qcr{SSS(z) + v = SSSSSSSS(y)}\]
			then $\bar{\psi} = \qcr{w = w}$ and
			\[\psi^* = \qcr{e_1 = e_2}.\] 
			The corresponding sequence of terms is
			\[\tau_{\psi} = \Seq{\qcr{SSS(z) + v}, \qcr{SSSSSSSS(y)}}\]
		\end{przyklad}

		So far we introduced four different languages: $\mathcal{L}$, $\mathcal{L}^+$, $\mathcal{L}^w$ and $\mathcal{L}^*$. Let us stress that the last two play only auxiliary roles in our reasoning and the first one was relevant only in the previous chapter and now is replaced by $\mathcal{L}^+$ as the ''basic'' language. Formulae $\term(x)$, $\Sent(x)$ etc. should be read as ''$x$ is a $\mathcal{L}^+$ term'' and ''$x$ is a $\mathcal{L}^+$ sentence'' (respectively). In particular by writing $\forall \phi$ we implicitly quantify over $\mathcal{L}^+$ sentences (compare Convention \ref{konw: glupia konwencja}).
		
		\begin{definicja}
			$\CT_0^+$ is the theory containing the following axioms
			\begin{itemize}
				\item $\PA^+$
				\item compositional axioms for $T$ for the language $\mathcal{L}^+$, as in Definition \ref{CT}.
				\item $\Delta_0-$induction for formulae of the language $\mathcal{L}^+\cup\{T\}$.
				\item \df{Generalized regularity principle}:
				\begin{equation}\label{genregularity}\tag{GREG}
				\forall \phi \forall x,y\ \ \bigl(\TermSeq(x)\wedge \TermSeq(y)\wedge  \val{x} = \val{y} \rightarrow T(\phi(x)) \equiv T(\phi(y)) \bigr)
				\end{equation} 
				
			\end{itemize}
			
		\end{definicja}
		
		By an adaptation of methods of Enayat--Visser from their proof of conservativity of $\CT^-+\eqref{II}$, one can show that $\CT^- + \eqref{II} + \eqref{genregularity}$ is conservative over $\PA$.

		\begin{lemat}\label{uwaga: ext t-t^*}
		
		$\CT_0^+\vdash\forall \phi \forall x,y \ \ \bigl(\TermSeq(x)\wedge y=\val{x} \rightarrow T(\phi(x))\equiv T(\phi(y^*))\bigr)$
	\end{lemat}
	\begin{proof}
			Using the definitions introduced above and Remark \ref{rem: t=valt*} one shows that provably in $\PA^+$ for every sequence $\sigma$ and every sequence of terms  $\tau$ such that $\sigma=\val{\tau}$,
			
			\[\val{\tau} = \sigma = \val{\sigma^*}\]
			hence by the Generalized regularity principle
			\[T(\phi(\tau))\equiv T(\phi(\sigma^*))\]
			which ends the proof.
		\end{proof}
		
		
		\begin{definicja}[$\PA^+$]
			Let $\phi$ be any $\mathcal{L}^*$ formula and let $x$ be any $\mathcal{L}^+$ variable not occurring in $\phi$ (either as free or a bounded one). Let $\tau_x$ be the sequence of terms satisfying the following conditions:
			\begin{enumerate}
				\item $\len(\tau_x) = \len(\tau_{\phi})$
				\item for every $i\in E(\phi)$, $(\tau_x)_i = \qcr{(x)_i}$
			\end{enumerate}
			We define $\phi_x = \Subst(\phi^*,\tau_x)$.
		\end{definicja}
		
		\begin{przyklad}
			Let $\phi = \qcr{\exists z\ \ SSS(z) + v = SSSSSSSS(y)}$ be the formula from Example \ref{przyk: compl, template}. Then $\phi^* = \qcr{\exists z\ \ SSS(z) + e_1 = e_2}$ and 
			\[\phi_x = \qcr{\exists z\ \ SSS(z) + (x)_0 = (x)_1}\]
		\end{przyklad}

		\begin{definicja}[$\PA^+$]
			The \df{index} of a formula $\phi$ is the maximum of its complexity and the greatest index of a variable occurring in it (either as free or a bounded one).
		\end{definicja}
		
		\begin{przyklad}
			Templates of sentences of index $\leq 3$ are precisely (codes of) the following ones:
			\begin{enumerate}
				\item $e_1 = e_2$, $\neg(e_1 = e_2)$
				\item $(e_1 = e_2)\wedge (e_3=e_4)$, $(e_1 = e_2)\vee (e_3=e_4)$
				\item $Q x_i (x_i = e_1)$, $Q x_i (e_2 = x_i)$, $Q x_i (e_1 = e_2)$, $Q x_i (x_i=x_i)$, for $i\leq 3$ and $Q\in\was{\forall,\exists}$.
			\end{enumerate}
		\end{przyklad}
		
		The following is the last technical definition in our paper.
		
		\begin{definicja}[$\PA^+$]
			Let $\phi^*$ be a template. Two sequences of terms $\tau,\sigma$ are said to be $\phi^*$-\df{equivalent} if for all $i$ such that $e_i$ occurs in $\phi^*$, $(\tau)_i = (\sigma)_i$. If $\tau$ and $\sigma$ are $\phi^*$ equivalent, we denote it by $\tau\sim_{\phi^*}\sigma$.
		\end{definicja}
		
		\begin{przyklad}
			Let $\phi = \qcr{\exists x (S(x) = 0\wedge y<z)}$. Then $\phi^* = \qcr{\exists x\ \ S(x) = 0\wedge e_1<e_2}$ and the following two sequences of terms are \emph{not} $\phi^*$ equivalent:
			\begin{enumerate}
				\item $\Seq{\qcr{S0}, \qcr{0}}$
				\item $\Seq{\qcr{0}, \qcr{0}}$
			\end{enumerate}
		\end{przyklad}
		
		\begin{uwaga}[$\PA^+$]\label{uwaga: glupia uwaga}
			If $\tau\sim_{\phi^*}\sigma$ then $\phi^*(\tau) = \phi^*(\sigma)$.
		\end{uwaga}
		
		\begin{uwaga}\label{uwaga: ext T}
			Let us observe that provably in $\PA^+$ we have: for any (arithmetical) formula $\phi$ and any sequence of terms $\tau$ it holds that
			\begin{equation}\label{1}
			\phi_x(\num{\tau}) =_{df} \Subst(\phi_x,\num{\tau}) = \Subst(\phi^*,\tau^*) =_{df}\phi^*(\tau^*)
			\end{equation}
			Hence, in $\CT_0^+$ we have: for every $\mathcal{L}^+$ sentence $\phi$, every sequence $\sigma$ and every sequence of terms $\tau$ such that $\tau\sim_{\phi^*}\tau_{\phi}$ (for the definition of $\tau_{\phi}$ see Definition \ref{def: dluga} point \ref{tefi}) and $\sigma=\val{\tau}$
			\begin{align*}
			T\phi&\equiv T\phi^*(\tau) && \textnormal{by definitions of $\phi^*$ and $t_{\phi}$ and Remark \ref{uwaga: glupia uwaga}}\\
			&\equiv T\phi^*(\sigma^*) && \textnormal{by Lemma \ref{uwaga: ext t-t^*}}\\
			&\equiv T\phi_x(\num{\sigma})&& \textnormal{by \eqref{1}}
			\end{align*}
			In particular for any $\tau\sim_{\phi^*}\tau'$ if $\sigma = \val{\tau}$ and $\sigma' = \val{\tau'}$ we have
			\[T\phi_x(\num{\sigma})\equiv T\phi_x(\num{\sigma'})\]
		\end{uwaga}
		
		After these comments we are ready to state and prove our second theorem. 
		
		\begin{tw}
			$\CT_0^+$ proves the Global Reflection Principle and the Axiom Soundness Property for $\PA$.
		\end{tw}
		\begin{proof}
			We work in $\CT_0^+$. The fact that $\CT_0^+$ proves that all axioms of $\PA^+$ are true is obvious. Let us show the Global Reflection Principle. 	Note that, provably in $\PA^+$, for every $c, d$ there are only finitely many templates for formulae of index less than $c$. Each such formula can be obtained from one of those finitely many templates by the procedure of substituting terms for the $e_i$'s. Let $\gamma(c)$ be a (code of a) set of all templates for sentences of index at most $c$. Let $y,z$ be any $\mathcal{L}^+$ variables which do not occur in those sentences. As in the main theorem we shall make use of simplistic truth predicates. But this time working in extended language enables us to make them even more simplistic. Let us put:
			\begin{eqnarray}
			T_c(x) := \bigvee_{\phi\in\gamma(c)}\bigl(\exists y,z\ \ (\TermSeq(y) \wedge x = \Subst(\phi,y) \wedge z=\val{y} \wedge \phi_z)\bigr)\nonumber
			\end{eqnarray}
			For example the disjuncts of $T_3(x)$ corresponding to the templates $\qcr{e_1=e_2}$ and $\qcr{\exists x_2(x_2 = e_1)}$ are the following
			\begin{equation}
			\bigl(\exists y,z\ \ (\TermSeq(y) \wedge x = \Subst(\qcr{e_1 = e_2},y) \wedge z=\val{y} \wedge (z)_0 = (z)_1)\bigr)\nonumber
			\end{equation}
			\begin{equation}
			\bigl(\exists y,z\ \ (\TermSeq(y)\wedge x = \Subst(\qcr{\exists x_2(x_2 = e_1)},y)\wedge z= \val{y} \wedge \exists x_2(x_2 = (z)_0))\bigr)\nonumber
			\end{equation}			
			Observe that having projection functions as primitive symbols in the language makes it possible to use finitely many quantifiers at the beginning of each disjunct of $T_c(x)$. In case of their absence we would have to add one quantifier for each variable $e_i$ of a template $\phi$, which in case of non-standard formulae would result in a block of quantifiers of non-standard length. As in the proof of the main theorem the function $c\mapsto T_c$ is primitive recursive. As in the proof of our main theorem, define
			\[T'_c(x) = T\bigl(T_c(\num{x})\bigr)\]
			Arguing exactly as in the proof of Lemma \ref{lemat_induktywnosc_TTc} we show that for every $c$ every formula with $T'_c$ satisfies the induction scheme. 
			This time our simplistic predicates have an additional feature: for any sentence $\phi$ of index less or equal to $c$ we have
			\[T(\phi)\equiv T'_c(\phi)\]
			Note that it follows from the above, that $TT_c$ are compositional on the sentences of index less than $c$. To prove the above assertion let us fix a sentence $\phi$ of index less than or equal to $c$ and observe that for $\sigma=\val{\tau_{\phi}}$ we have
			\begin{eqnarray}
			T\bigl(T_c(\num{\phi})\bigr) &\equiv& T \biggl( \bigvee_{\psi\in\gamma(c)}\bigl(\exists y, z\ \ (\TermSeq(y)\wedge \phi = \Subst(\psi,y)\wedge z=\val{y} \wedge \psi_z)\bigr) \biggr) \nonumber\\
			&\equiv& T(\phi_z(\num{\sigma}))\nonumber\\
			&\equiv& T(\phi)\nonumber
			\end{eqnarray}
			Indeed, the first equivalence holds by definition, and the third is obtained by Remark \ref{uwaga: ext T}. Let us focus on the second one. From left to right: if
			\[T\biggl(\bigvee_{\psi\in\gamma(c)}\bigl(\exists y,z\ \ (\TermSeq(y)\wedge z=\val{y}\wedge \phi = \Subst(\psi,y)\wedge\psi_z)\bigr)\biggr)\]
			holds then by Disjunctive Correctness of $T$ we get that for some template $\psi$ it holds that
			\[T\bigl(\exists y,z\ \ (\TermSeq(y)\wedge z=\val{y}\wedge \phi = \Subst(\psi,y)\wedge \psi_z)\bigr)\]  
			Hence for some $\tau',\sigma'$ such that $\sigma'=\val{\tau'}$ we have by compositionality of $T$
			\[\TermSeq(t) \wedge \phi = \Subst(\psi,\tau')\wedge T\psi_z(\num{\sigma'})\]
			But for $\psi\neq\phi^*$ this sentence can be easily disproved in $\CT^-$ (because the middle conjunct is disprovable already in $\PA^+$). Moreover it must be the case that $\tau'\sim_{\phi^*} \tau_{\phi}$. Hence, for $\sigma = \val{\tau_{\phi}}$ we get $T(\phi_z(\num{\sigma}))$ (invoking Remark \ref{uwaga: ext T}). From right to left it is easier: if, for $\sigma = \val{t_{\phi}}$ we have $T(\phi(\num{\sigma}))$ then also
			\[\TermSeq(\tau_{\phi}) \wedge \phi = \Subst(\phi^*,\tau_{\phi})\wedge \sigma=\val{\tau_{\phi}}\wedge T\phi_z(\num{\sigma})\] 
			Hence $T(\exists y,z\ \ \TermSeq(y)\wedge \phi = \Subst(\phi^*,y)\wedge z=\val{y}\wedge \phi_z)$ and once again by Disjunctive Correctness we get
			\[T \biggl( \bigvee_{\psi\in\tau(c)}\bigl(\exists y,z\ \ (\TermSeq(y)\wedge \phi = \Subst(\psi,y)\wedge z=\val{y} \wedge\psi_z)\bigr) \biggr)\]
			Now we proceed as in the proof of our main theorem. Let $d$ be a proof in First Order Logic of a sentence $\phi\in \mathcal{L}^+$ from true premises. There is a $c$ such that each formula $\psi$ occurring in $d$ is of index at most $c$. Then by induction on the length of $d$ we check that in each sequent $\Gamma\longrightarrow\Delta$ in $d$ if for every $\psi$ in $\Gamma$ we have $T'_c(\psi)$, then for some $\theta$ in $\Delta$ we have $T'_c(\psi)$. This is legitimate, since $T'_c(x)$ is inductive and compositional on formulae occurring in $d$. Hence $T'_c(\phi)$ and by the above considerations also $T(\phi)$.
		\end{proof}

		\section{Appendix}

		In the following section we shall discuss the proof by Kotlarski and the mistake that has been pointed out by Richard Heck and Albert Visser. We decided to include this discussion, since the alleged result has been cited, sometimes with repeating the erroneous proof, in at least three different papers, some of them written already after the gap in the proof has been observed. Let us present Kotlarski's argument.
		
		We would like to show that $\CT_0$ proves that all theorems of $\PA$ (axiomatised with a parameter-free induction scheme) are true. We know that $\CT_0$ proves that any instance of the parameter-free induction scheme for arithmetical formulae under any substitution of closed terms for free variables is true. It suffices to show that for any proof $d$ formalized in Hilbert system if all its premises are true, then the conclusion is true. But the only rule of Hilbert calculus, namely Modus Ponens, is clearly truth-preserving. In the Hilbert system, we assume only finitely many axiom schemes for first-order logic and arbitrary propositional tautologies. We may also easily check in $\CT^-$ that all these finitely many axiom schemes are true for arbitrary sentences. By $\Delta_0$-induction we may check that all sentences which are propositional tautologies are true. Thus every provable sentence is true.
		
		However Kotlarski overlooked a crucial detail in the formulation of \newline Hilbert calculus. We can assume that metavariables $\phi, \psi$ occurring in the axiom schemes for this calculus represent arbitrary arithmetical formulae, not necessarily sentences. In such a case we apparently have to add a generalization rule  to Hilbert calculus. Then if we want to show by induction that first-order derivations preserve truth, the actual induction thesis should  state something like: ''for any substitution of terms for free variables the resulting sentence is true'' which is a $\Pi_1$ statement. In another possible formulation of Hilbert calculus we may assume that metavariables $\phi,\psi$ occurring in its axiom schemes represent only arithmetical \emph{sentences}, but then we have to take as logical axioms arbitrary \emph{universal closures} of propositional tautologies rather than tautologies themselves. But then in turn we have to prove the following statement: universal closures of propositional tautologies are true. So consider any sentence of the form:
		\begin{displaymath}
		\forall x_1 \forall x_2 \ldots \forall x_c \ \ \phi(x_1, \ldots, x_c),   
		\end{displaymath}	    
		where $\phi(x_1, \ldots, x_c)$ is a propositional tautology. Using $\Delta_0$-induction for the truth predicate we may indeed prove that for arbitrary numerals $\num{a_1}, \ldots, \num{a_c}$ the following sentence is true:
		
		\begin{displaymath}
		\phi(\num{a_1}, \ldots, \num{a_c}).
		\end{displaymath}		   
		
		Then by compositional axioms we may even prove for any fixed standard $k$ (i.e. an element of true $\omega$, viewed externally) that the following is also true, where the block of the universal quantifiers is of standard length:
		
		\begin{displaymath}
		\forall x_{c-k} \ldots \forall x_c \ \ \phi(\num{a_1}, \ldots, \num{a_{c-k-1}}, x_{c-k}, \ldots, x_c).
		\end{displaymath}
		
		But to prove that the whole universal closure is true we apparently need to use induction over the following statement: ''For all $y<c$ and all numerals $\num{a_1}, \ldots, \num{a_{c-y}} $ the sentence  $\forall x_{c-y} \ldots \forall x_c \ \ \phi(\num{a_1}, \ldots, \num{a_{c-y-1}}, x_{c-y}, \ldots, x_c)$ is true''. And to do this we would need $\Pi_1$ induction, since we quantify over arbitrary numerals. 
		
		Moreover, it seems that the problem of nonstandard blocks of universal quantifiers is not a mere technicality. Note that taking universal closures of propositional tautologies in Hilbert calculus basically allows us to dispense of eigenvariables of sequent calculus and it seems reasonable that to prove that reasoning in sequent calculus preserves truth, we really need to quantify over all possible substitutions of terms for eigenvariables and thus basically we are forced to employ $\Pi_1$ induction for the truth predicate.

		\section{Acknowledgements}
		
		We are extremely grateful to our supervisor, Cezary Cieśliński for his invaluable support and for introducing us to the problem of the conservativity of $\Delta_0$-induction for the compositional truth predicate. We would also like to express our sincere gratitude to the anonymous referee, who spotted a substantial but subtle mistake in the first version of the paper. We are also grateful to Ali Enayat and Albert Visser for extremely interesting discussions. 
		
		The research presented in this paper was supported by the National Science Centre, Poland (NCN), grant number 2014/13/B/HS1/02892.

		\nocite{*}
		\bibliography{NotesCT2}

\end{document}